\RequirePackage{amsmath}
\documentclass[runningheads]{amsart}
\usepackage[foot]{amsaddr}
\usepackage{xcolor}
\definecolor{mgt}{RGB}{238,28,189}
\usepackage[T1]{fontenc}
\usepackage{graphicx}
\usepackage{amsmath,amssymb,amsthm}
\usepackage{xspace}
\usepackage{todonotes}
\usepackage{virginialake}
\usepackage{stmaryrd}
\usepackage[colorlinks]{hyperref}
\hypersetup{
    linkcolor=teal,
    filecolor=teal,
    citecolor=orange,
    urlcolor=cyan
}
\usepackage[capitalise]{cleveref}
\usetikzlibrary{arrows}
\usepackage{tikz-cd}

\usepackage{enumerate}

\newcommand{\TM}{\text{\rm\tt M}}

\title[Undecidability of theories of semirings with fixed points]{Undecidability of theories of semirings \\ with fixed points}
\author{Anupam Das$^1$}
\address{$^1$University of Birmingham, UK}
\email{$^1$a.das@bham.ac.uk}

\author{Abhishek De$^2$} 
\address{$^2$Krea University, India}
\email{$^2$abhishek.de@krea.edu.in}

\author{Stepan L.\ Kuznetsov$^3$}
\address{$^3$Steklov Mathematical Institute of RAS, Russia}
\email{$^3$skuzn@inbox.ru}

\date{\today}

\renewcommand{\phi}{\varphi}
\renewcommand{\emptyset}{\varnothing}
\renewcommand{\epsilon}{\varepsilon}
\newcommand{\defname}[1]{\textbf{{#1}}}
\newcommand{\df}{:=}

\newcommand{\Pow}{\mathcal P}
\newcommand{\pow}[1]{\Pow(#1)}

\newtheorem{theorem}{Theorem}
\newtheorem{lemma}[theorem]{Lemma}
\newtheorem{proposition}[theorem]{Proposition}

\crefformat{observation}{Observation~#2#1#3}

\newtheorem{corollary}[theorem]{Corollary}
\newtheorem{fact}[theorem]{Fact}

\theoremstyle{definition}
\newtheorem{definition}[theorem]{Definition}
\newtheorem{example}[theorem]{Example}
\newtheorem{remark}[theorem]{Remark}

\newcommand{\Nat}{\mathbb{N}}

\newcommand{\purple}[1]{{\color{purple}#1}}

\newcommand{\Alphabet}{\mathcal{A}}
\newcommand{\Var}{\mathcal{V}}

\newcommand{\proves}{\vdash}

\newcommand{\Lang}{\mathcal L}
\newcommand{\lang}[1]{\Lang(#1)}

\newcommand{\id}{\mathsf{init}}

\newcommand{\limp}{\rightarrow}

\newcommand{\basicax}{\mu\mathsf{S}}
\newcommand{\idem}{\mathsf{idem}}

\newcommand{\ind}{\mathsf{lfp}}

\newcommand{\N}{\mathbb N}

\newcommand{\regtot}{\mathbf R}
\newcommand{\tot}{\mathbf T}
\newcommand{\nottot}{\bar\tot}

\newcommand{\loops}{\mathbf{C}}
\newcommand{\halts}{\mathbf{H}}

\newcommand{\total}[1]{\mathsf{Tot}(#1)}

\newcommand{\basicaxidfree}{\basicax_{\not {\hspace{.1em}1}}}

\renewcommand{\id}[1]{#1^i}

\newcommand{\dom}[1]{\mathsf{dom}(#1)}

\newcommand{\prf}[1]{f_{#1}}
\newcommand{\res}[1]{W_{#1}}
\newcommand{\inv}[1]{#1^{-1}}

\begin{document}

\begin{abstract}
    In this work we prove the undecidability (and $\Sigma^0_1$-completeness) of several theories of semirings with fixed points. 
    The generality of our results stems from recursion theoretic methods, namely the technique of \emph{effective inseperability}.
    Our result applies to many theories proposed in the literature, including Conway $\mu$-semirings, Park $\mu$-semirings, and Chomsky algebras.
\end{abstract}

\maketitle

\section{Introduction}

The \emph{context-free languages} (CFLs) may be presented as the least solutions to certain systems of algebraic inequalities, by construing \emph{context-free grammars} (CFGs) as such. 
It is well known that the resolution of such simultaneous fixed points may be reduced to unary ones, by way of \emph{Beki\'c's Lemma}, \cite{Bekic84}.
Indeed already in the early '70s
Gruska observed that expressions for CFLs could be obtained by an extension of usual regular expressions by an `iteration' operator for computing (least) fixed points.
Nowadays these are known as \emph{$\mu$-expressions} \cite{Leiss92:ka-with-recursion,EsikLeiss02,EsikLeiss05}.\footnote{Here the notation `$\mu$' insists that the fixed points named in the syntax are \emph{least}, instead of, say, greatest or even arbitrary fixed points.}

Algebraic theories of regular expressions are based on (idempotent) semirings with an appropriate $*$ operation, for instance \emph{Kleene algebras} (cf., e.g., \cite{Kozen94}). 
Duly, analogous structures for $\mu$-expressions are based on (idempotent) semirings with \emph{fixed points}.
Similarly to regular expressions, in the last few decades
a number of formulations of semirings with {fixed points} have appeared throughout the literature in formal language theory, logic, and algebra \cite{Courcelle86,Leiss92:ka-with-recursion,EsikLeiss02,EsikLeiss05,Hopkins08:alg-approach-1,Hopkins08:alg-approach-2,EsikKuich12:mod-aut-theory,GHK13,Leiss16}.

However until now, as far as we know, the (un)decidability of most of these equational theories has remained open.
This is in stark contrast to the situation for Kleene algebras, over regular expressions, where the complexity of equational and Horn theories, with and without continuity, is now well understood (see, e.g., \cite{Kozen02:complexity-ka,DKPP19:ka-with-hyps}).
In this work we address this situation by showing undecidability (in fact $\Sigma^0_1$-completeness) for a range of theories of (not necessarily continuous) semirings with fixed points, including many of the ones occurring in the literature, in particular from \cite{Leiss92:ka-with-recursion,EsikLeiss02,EsikLeiss05,GHK13}.

\subsection{Previous work}
Lei\ss\ investigated several classes of idempotent semirings with fixed points, all of which satisfy the \emph{induction axioms} characterising $\mu x e(x)$ as the least prefixed point of the operation $x\mapsto e(x)$. 
This interpretation of least fixed points is recovered from the famous Knaster-Tarski fixed point theorem, and is the standard way to present axiomatic theories with forms of (co)induction.
\'Esik and Lei\ss\ showed that reduction of CFGs to Greibach normal form \cite{Greibach1965} can be validated in a purely equational theory of semirings with fixed points (not necessarily satisfying induction), \emph{Conway $\mu$-semirings} \cite{EsikLeiss02,EsikLeiss05}. 

Grathwohl, Henglein and Kozen gave a sound and complete (infinitary) axiomatisation of CFL inclusions by way of $\mu$-expressions \cite{GHK13} (note that this problem is $\Pi^0_1$-complete). 
In particular they introduced the notion of a \emph{Chomsky algebra}, an algebraically closed idempotent semiring.
These turn out to be nothing more than idempotent Park $\mu$-semirings of \cite{Leiss92:ka-with-recursion,EsikLeiss02,EsikLeiss05}, where $\mu$-expressions give a naming convention for solutions to polynomial systems.
\cite{GHK13} showed that the theory of \emph{$\mu$-continuous} Chomsky algebras is sound and complete for CFL inclusions.

Later further metalogical results were established for Chomsky algebras and friends.
For instance Lei\ss\ showed that the matrix ring of a $\mu$-continuous Chomsky algebra is again $\mu$-continuous.
Two of the authors of the present work recovered the completeness result of \cite{GHK13} via proof theoretic means, even establishing a stronger `cut-free' completeness result \cite{DasDe24:omega-cfls}. 
In the same work they also proposed \emph{non-wellfounded} systems for both CFLs, via $\mu$-expressions, and an extension to $\omega$-words, using \emph{greatest fixed points}.

\subsection{Contributions}
In this work we show undecidability, in fact $\Sigma^0_1$-completeness, for a range of theories of semirings with fixed points. We give a basic equational axiomatisation $\basicaxidfree$ of semirings (not necessarily with identity) with a $\mu$ operator denoting (not necessarily least) fixed points.
In particular this theory is more general (i.e.\ weaker) than the Conway $\mu$-semirings, Park $\mu$-semirings and Chomsky algebras from earlier works.
Our main result is that \emph{any} recursively enumerable extension of $\basicaxidfree$ sound for the usual language model is undecidable, and in fact $\Sigma^0_1$-complete.
In particular Conway $\mu$-semirings, Park $\mu$-semirings and Chomsky algebras have undecidable equational theories.

To achieve such generality we employ methods from recursion theory, exploiting the technique of \emph{effective inseparability}. 
These methods were successfully employed in earlier works by one of the authors, in particular showing the undecidability of the equational theory of action lattices \cite{Kuznetsov19:action-logic,Kuznetsov2021TOCL}, a longstanding open problem, and the same for commutative action lattices \cite{Kuznetsov2023JLC}. 

\subsection{Structure of the paper}
In \cref{sec:prelims} we introduce a basic theory $\basicax$ of semirings (with identity) with fixed points, as well as other theories from the literature. 
In \cref{sec:eff-insep} we prove the effective inseparability of two problems on CFGs, `non-totality' $\nottot$ and `regular totality' $\regtot$, by reduction to a more well known effectively inseparable pair, `halting' and `looping' Turing machines.
In \cref{sec:undec-with-idem} we show the undecidability of any r.e.\ extension of $\basicax + {}$\emph{idempotency}, $e + e = e$, sound for the language model.
In \cref{sec:undec-without-idem} we improve this result to theories without identity and idempotency, by a sort of interpretation of idempotency via fixed points. 
Finally we make some concluding remarks in \cref{sec:concs}.

\newpage






\section{Some theories of semirings with fixed points}
\label{sec:prelims}

Let $\Alphabet$ be a finite alphabet and $\Var$ be a countably infinite set of variables disjoint from $\Alphabet$. 
We consider \defname{$\mu$-expressions} $e,f,$ etc.\ generated by the following grammar:
\[
e,f,\dots \quad ::= \quad
0 \quad \mid \quad 1 \quad \mid \quad a \in \Alphabet \quad \mid \quad x\in\Var \quad \mid \quad e+f \quad \mid \quad e\cdot f \quad \mid \quad \mu x e
\]
We usually just write $ef$ instead of $e\cdot f$.

Structures interpreting $\mu$-expressions, formally, must be defined appropriately taking account of variable binding and compatible with substitution. 
A formal exposition can be found in e.g.\ \cite{EsikLeiss05,Leiss16}, so here we give only examples:

\begin{example}
    [Languages]
    \label{language-model}
    Temporarily expand the syntax of expressions to include each language $A\subseteq \Alphabet^*$ as a new constant symbol.
    We associate to each closed $\mu$-expression $e$ (of this expanded syntax) a language $\lang e \subseteq \Alphabet^*$ as follows:
    \[
    \begin{array}{r@{\ \df \ }l}
         \lang 0 & 0 \\
         \lang 1 & \{\epsilon\} \\
         \lang a & \{a\} \\
         \lang A & A\\
         \lang {e+f} & \lang e \cup \lang f \\
          \lang {ef} & \{ uv : u \in \lang e \ \& \ v \in \lang f\} \\
         \noalign{\medskip}
          \lang {\mu x e(x)} & \bigcap\{ A : \lang {e(A)} \subseteq A \}
    \end{array}
    \]
\end{example}
Every symbol is interpreted by $\lang\cdot$ as a monotone operation, with respect to $(\pow {\Alphabet^*}, \subseteq)$.
The \emph{Knaster-Tarski} fixed point theorem tells us that monotone operations on complete lattices have least fixed points,\footnote{Indeed, they further have a complete sublattice of fixed points, in particular including a greatest fixed point.} in particular with the least fixed point of $A \mapsto e(A)$ given by $\lang e $ defined above: the intersection of all \emph{prefixed points} of $A \mapsto e(A)$.

It is not hard to see that the image of $\lang \cdot$ on $\mu$-expressions is just the context-free languages (CFLs) \cite{Gruska71:cfls-by-mu-exprs}.
This structure will form a model of all theories considered herein, and will guide our main undecidability results.
Let us point out already the following consequence:

\begin{proposition}
[See, e.g., \cite{HopcroftUllman79:book,Sipser96:book}]
    Deciding $\lang e \subseteq_? \lang f$ is $\Pi^0_1$-complete.
\end{proposition}

\subsection{Semirings with fixed points}
For generality of our results, we shall work with a minimal theory of semirings with fixed points, before considering specific examples such as (idempotent) Park $\mu$-semirings and Conway $\mu$-semirings later.

\begin{definition}
[Basic axioms]
    The equational theory $\basicax$ of $\mu$-expressions is given by all the following axioms:
    \begin{enumerate}[(i)]
    \item $(0,1,+,\cdot) $ form a semiring:
    \[
    \begin{array}{r@{\ = \ }l}
         e + (f+g) & (e+f)+g \\
         e+f & f + e\\
         e+0 & e 
    \end{array}
    \quad
    \begin{array}{r@{\ = \ }l}
         e (fg) & (ef)g \\
         e1 & e \\
         1 e & e \\
         e0 & 0 \\
         0e & 0
    \end{array}
    \quad
    \begin{array}{r@{\ = \ }l}
         e(f+g) & ef + eg \\
         (e+f)g & eg + fg
    \end{array}
    \]
    \item\label{item:fixed-point-axioms} $\mu x e(x)$ is a fixed point of $ x \mapsto e(x)$:
    \[
    \mu x e(x) = e (\mu x e(x))
    \]
    \end{enumerate}
\end{definition}

Note that $\basicax$ does \emph{not} insist that $\mu x e(x)$ is a \emph{least} fixed point, which would require non-equational axioms. 
Instead, notice that the axiomatisation above is finite and equational.

\begin{remark}
    [On $\mu$-congruence]
    \label{rem:mu-congruence}
    Note that $\basicax$ does not admit the \textbf{$\mu$-congruence} rule: $\vlinf{}{}{\mu x e(x) = \mu x f(x)}{e(x) = f(x)}$.
    This is often built into the definition of various classes of `$\mu$-semirings' in the literature \cite{EsikLeiss05,GHK13,Leiss16}, but we shall not exploit this here.
    This means that, even if two operators are (provably) equivalent, the fixed points named by their closure under $\mu$ may be different. 
    
    For instance we do not have in $\basicax$ that $\mu x (a+x) = \mu x (x+a)$.
    This is consistent with the fact that we are not axiomatising $\mu x e(x)$ to be a \emph{least} fixed point, but just \emph{some} fixed point.
    So, e.g., we can tweak the model $\Lang$ to satisfy $\mu x (a+x) =a $ but $\mu x (x+a) = a+b$.
    Of course only the former is the \emph{least} fixed point, but the resulting structure remains a model of $\basicax$.

    Construing expressions in a higher-order syntax, where $\mu$ is a higher-order constant rather than a binder, the lack of $\mu$-congruence may be seen as saying that $\mu$ does not behave \emph{extensionally}.
    In practice, in this work, this means we must be careful to not carry out reasoning \emph{under} a $\mu$ unless justified.
\end{remark}

\begin{remark}
    [Comparison to $\mu$-semirings]
    A \defname{$\mu$-semiring}, introduced in \cite{EsikLeiss05}, is defined just like $\basicax$ except:
    \begin{itemize}
        \item $\mu$-semirings must satisfy \emph{strong} $\mu$-congruence:
        $\forall x\,  e(x) = f(x) \limp \mu x e(x) = \mu x f(x)$; and,
        \item $\mu$-semirings do not require $\mu x e(x) $ to be a fixed point, i.e.\ axiom \eqref{item:fixed-point-axioms} is omitted.
    \end{itemize}
    Thus our system $\basicax$ is, strictly speaking, incomparable with general $\mu$-semirings, though it will still underlie extensions considered later.
     $\basicax$ should rather be considered a basic theory of `semirings with fixed points', namely the theory of those $\mu$-semirings in which $\mu x e(x)$ always denotes a fixed point of $x \mapsto e(x)$, albeit not necessarily the \emph{least} fixed point, and albeit not extensionally so (see previous \cref{rem:mu-congruence}).

     We omit $\mu$-congruence for two main reasons: (a) we do not need it to obtain our ultimate undecidability results, rendering them strictly more general; and (b) the axiomatisation $\basicax$ remains purely equational, and so could be of independent interest.
\end{remark}

We shall soon present $\mu$-semirings equipped with an order, so as to axiomatise the \emph{leastness} of fixed points denoted by $\mu$, in particular satisfying the `strong' congruence principle mentioned in the example above.
Before that let us recall a class of models of $\basicax$ that is covered by our results:

\begin{definition}
[\cite{EsikLeiss05}]
    A \defname{Conway $\mu$-semiring} is a $\mu$-semiring further satisfying:
    \begin{enumerate}
    \item\label{item:fp-partial-unfolding} $\mu x e(f(x)) = e(\mu x f(e(x)))$
        \item\label{item:fp-diagonal} $\mu x \mu y e(x,y) = \mu x e(x,x)$
    \end{enumerate}
\end{definition}

Note that the fixed point axiom \eqref{item:fixed-point-axioms} is already a consequence of \eqref{item:fp-partial-unfolding} above, by setting $f(x)\df x$, and so Conway $\mu$-semirings are also models of $\basicax$.
Conway $\mu$-semirings were introduced in \cite{EsikLeiss05} as a minimal theory in which certain normal forms of context-free grammars may be formally obtained, in particular \emph{Greibach} normal form. 
We shall not develop its metatheory any further here, as our results about Conway $\mu$-semirings depend only on the fact that their equational theory is recursively enumerable, extends that of $\basicax$, and is modelled by the structure of languages $\Lang$ from \cref{language-model}.

\subsection{Equational systems and context-free grammars}
Before considering further extensions of $\basicax$, let us revisit the interpretation of $\mu$-expressions as context-free languages.
It is known that Conway $\mu$-semirings can `solve' context-free grammars (CFGs), but we shall recast (and improve) this result within $\basicax$ here.

An \defname{equational system} is a set $\mathcal E (\vec X) = \{X_i = f_i(\vec X)\}_{i<n}$.
 By recursively applying \emph{Beki\'c's Lemma} (\cite{Bekic84}) we can solve such systems in $\basicax$:

\begin{proposition}
\label{eq-systems-have-solns-in-basicax}
    Fix an equational system $\mathcal E(\vec X) = \{X_i = f_i(\vec X)\}_{i<n}$. 
    There are (closed) expressions $\vec e $ s.t.\ $\basicax \proves \mathcal E(\vec e)$. 
\end{proposition}
\begin{proof}
    By induction on $n$:
    \begin{itemize}
        \item When $n=0$ the statement is vacuously true.
        \item Fix an equational system $\{X_i = f_i(\vec X, X_n)\}_{i\leq n}$, where $\vec X = X_0, \dots, X_{n-1}$.
        Write $e_n'(\vec X) \df \mu X_n f_n (\vec X, X_n)$.
        Now, by inductive hypothesis, let $\vec e$ be solutions of $\{X_i = f_i(\vec X, e_n'(\vec X))\}_{i<n}$, i.e.\ such that $\basicax \proves e_i = f_i(\vec e, e_n'(\vec e))$, for all $i<n$.
        It suffices now to show that $e_n \df e_n'(\vec e) $ is a solution for $X_n$, i.e.\ $\basicax \proves e_n = f_n (\vec e, e_n)$, for which we reason as follows:
        \[
        \begin{array}{r@{\ = \ }ll}
             e_n & e_n' (\vec e) & \text{by definition of $e_n$} \\
                & \mu X_n f_n (\vec e, X_n) & \text{by definition of $e_n'$} \\
                & f_n (\vec e, \mu X_n f_n (\vec e, X_n)) & \text{by fixed point axiom \eqref{item:fixed-point-axioms}} \\
                & f_n (\vec e, e_n) & \text{by definition of $e_n$} \qedhere
        \end{array}
        \]
    \end{itemize}
\end{proof}

Recalling \cref{rem:mu-congruence}, note that solutions to an equational system are not necessarily unique, as $\basicax$ does not insist that $\mu$ computes \emph{least} fixed points.

Of course, we can use the above result to `solve' CFGs too.

\begin{corollary}
\label{cfgs-have-solutions-in-basicaxs}
    Fix a CFG $\{X_i \Rightarrow f_{ij} (\vec X)\}_{i<n, j}$.
    There are $\mu$-expressions $\vec e = e_0, \dots, e_{n-1}$ s.t.\ $\basicax \proves e_i = \sum_j f_{ij} ({\vec e})$ for all $i<n$.
\end{corollary}

\subsection{Ordered semirings with fixed points}
Let us continue to recall certain classes of semirings over $\mu$-expressions from the literature.

    A semiring is \defname{idempotent} if it satisfies:
    \[
    \idem\ : \quad  e + e = e
    \]
We shall write $e\leq f \df e+f = f$,
typically in the presence of idempotency.
Note $\leq$, sometimes called the `natural order', defines a partial order under $\basicax + \idem$, and moreover one over which the semiring operations are monotone:

\begin{proposition}
\label{props-of-nat-order}
    The following hold in all idempotent semirings, and so in particular are consequences of $\basicax + \idem$:
\begin{equation}
    \label{eq:mon-po}
    \begin{array}{c}
     i\leq i \\
     e \leq f \  \& \ f \leq e \implies e = f \\
     e \leq f \ \& \  f \leq g \implies e \leq g
\end{array}
\qquad
\begin{array}{l}
e\leq f \implies e+i \leq f+i \\ 
e\leq f \implies eg \leq fg \\
e \leq f \implies ge \leq gf
\end{array}
\end{equation}
Moreover these properties hold in \emph{any} semiring, as long as $i$ is idempotent, and in particular are consequences of $\basicax + (i+i=i)$.
\end{proposition}
\begin{corollary}
\label{intro-elim-props-for-+}
    The following hold in all idempotent semirings, and so in particular are consequences of $\basicax + \idem$:
    \begin{equation}
        \label{eq:props-of-nat-ord}
    \end{equation}
    Moreover these properties hold in \emph{any} semiring, as long as $i$ is idempotent, and in particular are consequences of $\basicax + (i+i=i)$.
\end{corollary}

\begin{remark}
    [Ordered $\mu$-semirings]
    An \defname{ordered $\mu$-semiring} is a $\mu$-semiring equipped with a partial order satisfying the properties from \cref{eq:mon-po} above, as well as an ordered analogue of $\mu$-congruence, \defname{$\mu$-monotonicity}: $\vlinf{}{}{\mu x e(x) \leq \mu x f(x)}{e(x) \leq f(x)}$ (see \cite{EsikKuich12:mod-aut-theory,EsikLeiss02} for further details).
    
    Again, this does not necessarily hold in models of $\basicax+\idem$, even ones closed under (strong) $\mu$-congruence.
    For instance if $e(x) $ is $a+x$ and $f(x) $ is $a + b + x$, then $e(x)\leq f(x)$ is valid in all models of $\basicax+\idem$, but $\mu x e(x)$ needs only be greater than $a$ and $\mu x f(x) $ needs only be greater than $a+b$.
    Again tweaking the language model, it is consistent to interpret $\mu x e(x) $ as $\{a,b,c\}$ but $\mu x f(x)$ as $\{a,b\}$.
    Note that, since $e(x) \leq f(x) $ is valid in \emph{all} models of $\basicax + \idem + \mu$-congruence, it follows that $\mu$-monotonicity is not even admissible for this theory.
\end{remark}

In this work we shall only consider the `natural order' already defined, $e\leq f \df e+f = f$, before more generally structures without idempotency.
It is worth recalling at this point:
\begin{fact}
[Idempotent ideal]
\label{rem:idempotents-form-ideal}
The class of idempotent elements of a semiring always forms an ideal.
\end{fact}

In the presence of an order it is of particular interest to examine structures where $\mu$ calculates a \emph{least} (pre)fixed point wrt.\ the order,
\[
\ind\ :\quad  
e(\mu x e(x)) \leq \mu x e(x)
\qquad \& \qquad
e(f)\leq f \implies \mu x e(x) \leq f
\]
yielding an induction principle.
Note that the axioms above are neither (in)equational, nor finitely many: there is an axiom for each choice of $f$ above.

\begin{definition}
        A \defname{Park $\mu$-semiring} is a $\mu$-semiring satisfying $\ind$. A \defname{Chomsky Algebra} is an idempotent Park $\mu$-semiring whose order is the natural order.\footnote{In fact the final clause here is redundant: the \emph{only} monotone partial order on an idempotent semiring is the natural order.}
\end{definition}

Chomsky Algebras were defined in \cite{GHK13,Leiss16} as \emph{algebraically complete} idempotent semirings, while Park $\mu$-semirings were defined axiomatically in \cite{EsikLeiss02,EsikLeiss05}. 
Both presentations are equivalent, when construing $\mu$-expressions as a naming convention for equational systems (see previous subsection), as was noticed already in \cite{GHK13}.
Note that the fixed point axiom \eqref{item:fixed-point-axioms} is already a consequence of $\ind$ (see \cite{EsikLeiss02,EsikLeiss05}), and so Chomsky Algebras are models of $\basicax+\idem$, and Park $\mu$-semirings are models of $\basicax$.


Again, we shall not recount the established metatheory of (idempotent) Park $\mu$-semirings or Chomsky Algebras: what is important for our results are that they are recursively enumerable extensions of $\basicax (+ \idem)$, that are still sound for $\Lang$.

\section{Effective inseparability of two problems on CFGs}
\label{sec:eff-insep}

Before presenting our main results on the undecidability of various theories of (idempotent) semirings with fixed points, we shall take a detour through the recursion theory of problems on CFGs.

\subsection{Two problems on CFGs}

For our arguments, we shall use {\em effective inseparability} of two algorithmic problems for context-free grammars (CFG). 
We shall consider only a special class of CFGs which we call \emph{productive} ones. Let $\Alphabet$ be the set of all terminal symbols and let $\Var$ be the set of all non-terminal ones.
As usual we shall assume among the non-terminals a designated \textbf{start} symbol $S \in \Var$.

\begin{definition}
[Productive CFGs]
A CFG is \textbf{productive} if, for every production rule $X \Rightarrow \alpha_1 \ldots \alpha_n$, some $\alpha_i$ is a terminal, i.e.\ some $\alpha_i \in \Alphabet$.
\end{definition}

A typical example of productive CFG is a Greibach grammar without $\varepsilon$-pro\-duc\-ti\-ons. In Greibach grammars, we always have $\alpha_1 \in \Alphabet$. Thus, by Greibach's normalisation theorem~\cite{Greibach1965}, one shows that any context-free language without the empty word is generated by a productive CFG. On the other hand, the empty word may never be generated by such a grammar.
In the sequel we shall restrict our attention to productive CFGs.

The first problem we are going to consider is the {\em totality} problem (i.e., the question whether a given productive CFG generates all non-empty words over $\Alphabet$). The second one is more specific. It could be called ``the grammar is total for obvious reasons,'' or {\em regularly total}. 
Formally:


\begin{definition}
[(Regular) totality]
\label{def-regtotal}
A productive CFG is \textbf{total}, if it generates (from $S$) all non-empty words over $\Alphabet$.


A productive CFG $G$ is \textbf{regularly total}, if it includes two designated non-terminals $Y,T \in \Var$ such that:
\begin{itemize}
    \item For each $a \in \Alphabet$, $G$ includes all productions of the form:
    \begin{align*}
& S \Rightarrow a Y T \\
& T \Rightarrow a \\
& T \Rightarrow a T 
\end{align*}
\item There is some $n>0$ such that:
\begin{enumerate}
    \item $S$ generates all nonempty words of length $\leq n+1$;
    \item $Y$ generates all words of length $n$.
\end{enumerate}
\end{itemize}

\end{definition}

It is not hard to see that
every regulary total productive CFG is total: 

\begin{proposition}
    [Regular totality implies totality]
    Every regularly total productive CFG is total.
\end{proposition}
\begin{proof}
    Let $G$ be a regularly total productive CFG and let $\vec a$ be a nonempty word:
    \begin{itemize}
        \item If $|\vec a| \leq n+1$ then it is generated from $S$ by definition of regular totality.
        \item If $|\vec a | > n+1$ then write $\vec a = a\vec b \vec c$, for some $a \in \Alphabet$, $|\vec b| = n$ and $|\vec c| > 0 $. 
        Then $S \Rightarrow a YT$ and $Y$ generates $\vec b$, by assumption, and $T$ generates $\vec c$ (since it trivially generates \emph{every} nonempty word). \qedhere
    \end{itemize}
\end{proof}

The algorithmic complexity of regular totality is dual to that of totality: while the latter is known to be $\Pi^0_1$-complete, the former belongs to $\Sigma^0_1$ (and, as we shall see below, is also complete in this class).

In what follows, we tacitly identify productive CFGs with their G\"odel numbers under a reasonable encoding. 
We write $\regtot$ for the set of (G\"odel numbers of) productive CFGs which are regularly total and $\nottot$ for the set of (G\"odel numbers of) productive CFGs which are not total.

\subsection{Effective inseperability by reduction}

We recall, following~\cite{Kuznetsov2021TOCL}, the notion of effective inseparability and the results which allows it to be used for proving $\Sigma^0_1$-completeness of logical theories. 

Fix some acceptable enumeration $\{ \prf n\}_{n\in \N}$ of the one-input partial recursive functions on $\N$, 
and let us write $\res n$ for ``the $n$\textsuperscript{th} r.e.\ set,'' i.e.\ $\dom { \prf n}$, the domain of $ \prf n$.

\begin{definition}
$A,B\subseteq \N$ are \textbf{effectively inseparable} if:
\begin{itemize}
    \item $A \cap B = \emptyset$; and,
    \item there is a partial recursive function $h(x,y)$ s.t., whenever $\res u \supseteq A$ and $\res v \supseteq B$ with $\res u \cap \res v = \emptyset$, we have $h(u,v) \downarrow\,  \notin \res u \cup \res v$. 
\end{itemize}
We shall call such $h(x,y)$ a \textbf{productive function} for $A,B$.
\end{definition}

Given $A,B \subseteq \N$, a set $C\subseteq \N $ \textbf{separates} $A$ from $B$ (or is a \textbf{separator}) if $C\supseteq A$ but nonetheless $C\cap B = \emptyset$. 
In the Definition above, informally, $h$ uniformly falsifies any attempt to separate $A$ from $B$ by a decidable set.
Intuitively, we can see $\lambda v\,  h(u,v)$ here as a `proof' that any such $\res u$ is not decidable: it is a recursive function that computes counterexamples to the statement that $\res u$ has an r.e.\ complement.

The following result is a corollary of a theorem by Myhill~\cite{Myhill1955}, as shown in~\cite[Exercise 11-14]{Rogers1987}:
\begin{theorem}
\label{separator-of-eff-insep-sets-is-S01-complete}
If $A,B \subseteq \N $ are effectively inseparable r.e.\ sets and $C$ is an r.e.\ set separating them then $C$ is $\Sigma^0_1$-complete.
\end{theorem}

Note in the Theorem above, and generally in the sequel, we shall not be precise about whether $C$ separates $A$ from $B$ or vice-versa, WLoG by symmetry of the statement.

Effective inseparability can be propagated by a (very weak) kind of reduction.
Given $A,B\subseteq \Nat$ with $A\cap B = \emptyset$ and $A', B'\subseteq \Nat$ with $A' \cap B' = \emptyset$, an \textbf{e.i.-reduction} from $(A',B')$ to $(A,B)$ is a total computable function $f$ such that: 
\begin{enumerate}
\item\label{item:f(A)-in-A'} $f(A) \subseteq A'$
\item\label{item:f(B)-in-B'} $f(B) \subseteq B'$
\end{enumerate}
Such maps are indeed appropriate for reducing questions of effective inseperability:

\begin{proposition}\label{prop-reduction}
Suppose $A,B \subseteq \N$ are effectively inseparable, and 
 $A',B' \subseteq \N$ with $A' \cap B' = \emptyset$.  
If $f$ is an e.i.-reduction from $(A',B')$ to $(A,B)$ 
then $A'$ and $B'$ are also effectively inseparable.
\end{proposition}

\begin{proof}
The preimage of an r.e.\ set under a total computable mapping is itself r.e.
Indeed we have $\inv f (\res n) = \dom {\prf n \circ f}$, and we can even effectively compute from $n$ the index of the r.e.\ function $\prf n\circ f $, say by a total recursive function $g$, so that always $\inv f(\res n) = \res {g(n)}$.

Now, let $h$ be a productive function for $A,B$ and let us define a productive function $h'$ for $A',B'$ as follows: \[h'(u,v) = f(h(g(u), g(v)))\] 

Clearly $h'$ is a partial recursive function.
To show it is productive for $A',B'$, suppose
  $\res u \supseteq A'$, $\res v \supseteq B'$ with
  $\res u \cap \res v = \emptyset$ and let us show that $h'(u,v) \notin \res u \cup \res v$. 
First we have:
\[
\begin{array}{rcll}
     \res{g(u)} & = & \inv f (\res u ) & \text{by definition of $g$} \\
        & \supseteq & \inv f (A') & \text{since $\res u \supseteq A'$} \\
        & \supseteq & A & \text{since $f(A) \subseteq A'$} \\
    \noalign{\medskip}
    \res{g(v)} & = & \inv f (\res v ) & \text{by definition of $g$} \\
        & \supseteq & \inv f (B') & \text{since $\res v \supseteq B'$} \\
        & \supseteq & B & \text{since $f(B) \subseteq B'$} \\
    \noalign{\medskip}
    \res{g(u)} \cap \res{g(v)} & = & \inv f (\res u) \cap \inv f (\res v) & \text{by definition of $g$} \\
        & = & \inv f (\res u \cap \res v ) & \text{since preimages preserve intersections}\\
        & = & \inv f (\emptyset) & \text{since $\res u \cap \res v = \emptyset$} \\
        & = & \emptyset
\end{array}
\]



From here we have:
\[
\begin{array}{rcll}
     h(g(u),g(v)) & \notin & \res{g(u)} \cup \res{g(v)} & \text{since $h$ is productive for $A,B$} \\
        & \notin & \inv f (\res u) \cup \inv f (\res v) & \text{by definition of $g$} \\
        & \notin & \inv f (\res u \cup \res v) & \text{since preimages preserve unions}
\end{array}
\]
%
Hence $h'(u,v) = f(h(g(u), g(v)) \notin  \res u \cup \res v$, as required.
\end{proof}

\subsection{Effective inseperability of $\regtot$ and $\nottot $}

The main result of this section is the following:

\begin{theorem}
\label{regtot-nottot-eff-insep}
The sets $\regtot$ and $\nottot$ are effectively inseparable.
\end{theorem}
This subsection is devoted to proving this result.
We proceed via a reduction from a more well-known pair of effectively inseparable problems: halting and looping for Turing machines. 
Throughout this section we shall work with one-input deterministic Turing machines that include, besides the initial and final state, also a designated {\em looping} state $q_c$ with rules that force the machine to remain at this state once it is reached: $(q_c, a) \Rightarrow (q_c, a, N)$ for any letter $a$ of the tape alphabet.

We say that a machine $\TM$ loops on an input $x$ if $\TM$ reaches $q_c$ when running on $x$. Looping is a very specific case of non-halting. Let 
\begin{align*}
& \loops = \{ (\TM,x) \mid \mbox{$\TM$ loops on input $x$} \}, \\
& \halts = \{ (\TM,x) \mid \mbox{$\TM$ halts on input $x$} \}
\end{align*}
Again, we construe both these sets as subsets of $\Nat$ by implicitly identifying pairs $(\TM,x)$ with their codes.
The following fact is folklore (see, e.g., \cite[Exercise 7-55d]{Rogers1987}\cite[Proposition 3]{Kuznetsov2021TOCL}):

\begin{proposition}\label{prop-insepCH}
The sets $\loops$ and $\halts$ are effectively inseparable.
\end{proposition}

Now we can prove the effective inseparability of $\regtot$ and $\nottot$ by reduction, in the sense of Proposition~\ref{prop-reduction}, to the effective inseparability of $\loops$ and $\halts$.

Let us start by briefly recalling some ideas from the standard proof of $\Pi^0_1$-hardness for the totality problem for CFGs, as presented in Kozen's textbook~\cite{Kozen1997}. 
The {\em halting protocol} of a Turing machine $\TM$ on a given input $x$ (if it exists) is the word $\# \kappa_0 \# \kappa_1 \# \ldots \# \kappa_n \#$, where $\kappa_0$ is the initial configuration $q_0 x$, each $\kappa_{i+1}$ is the configuration which is the immediate successor of $\kappa_i$, and $\kappa_n$ is a final configuration. Let $\Alphabet$ be the alphabet in which configurations are written down (which includes the internal alphabet of $\TM$, its set of states, and the additional symbol $\#$).
We may readily define a (computable) {reducing function}, mapping each pair $(\TM,x)$ to a CFG which generates all non-empty words except for the halting protocol (if it exists). Thus, this grammar is total if and only if $\TM$ does not halt on $x$.

We shall modify this construction a bit, following~\cite{Kuznetsov2021TOCL}, in order to take care of looping. 
All non-empty words which are not the halting protocol must belong to one of the following three classes:
\begin{enumerate}[(I)]
\item\label{not-prefix-haltprot} words of length greater than 1, beginning with $\#$, which cannot be even a prefix of the halting protocol: this includes words not beginning with $\# \kappa_0 \#$; words where there is not a correct configuration between two $\#$'s; words with subwords $\# \kappa \# \kappa' \#$ where $\kappa'$ is not the immediate successor of $\kappa$; and, most importantly, all words including $q_c$;
\item\label{incomplete-prot} words of length greater than 1, beginning with $\#$, and not ending with $\# \kappa_F \#$ where $\kappa_F$ is a final configuration (such a word could be an incomplete protocol---a prefix of the halting protocol);
\item\label{begins-incorrectly} words beginning with a letter different from $\#$, and the one-letter word $\#$.
\end{enumerate}
Note that these classes are not necessarily disjoint: a word can belong to more than one of them.

Write (I)$'$ and (II)$'$ for the sets of words from (I) and (II), respectively, with the first symbol $\#$ removed.
The standard construction shows that both (I)$'$ and (II)$'$ are context-free. Denote the grammars generating (I)$'$ and (II)$'$ by $G_1$ and $G_2$, respectively; we suppose that both grammars are productive (e.g., being Greibach grammars). Here it is important that (I)$'$ and (II)$'$ are $\varepsilon$-free. Moreover, $G_1$ and $G_2$ are computable from $(\TM,x)$~\cite[Section~2.1]{Kuznetsov2021TOCL}. 

We assume $G_1$ and $G_2$ have disjoint sets of non-terminals, writing $Y$ and $Z$ for their start symbols, respectively. 
Now define the grammar $G_{\TM,x}$ obtained from the union of $G_1$ and $G_2$ along with the following additional productions:
\begin{align*}
& S \Rightarrow aYT && \mbox{ for each $a \in \Alphabet$} \\
& T \Rightarrow a T && \mbox{ for each $a \in \Alphabet$} \\
& T \Rightarrow a && \mbox{ for each $a \in \Alphabet$} \\
& S \Rightarrow \# Y \\
& S \Rightarrow \# Z \\
& S \Rightarrow a T && \mbox{ for each $a \in \Alphabet - \{ \# \}$} \\
& S \Rightarrow a && \mbox{ for each $a \in \Alphabet$}
\end{align*}
Here $S$  and $T$ are fresh non-terminals, and $S$ is the new start symbol of $G_{\TM,x}$.
Notice that the first three productions are exactly as in \cref{def-regtotal} for regular totality.

\begin{proposition}
\label{grammar-generates-all-but-haltprot}
    $G_{\TM,x}$ generates all nonempty words over $\Alphabet$, except the halting protocol of $\TM$ on $x$ (if it exists).
\end{proposition}
\begin{proof}
For words from (I) or (II) we start by firing $S \Rightarrow \# Y$ and $S \Rightarrow \# T$, respectively, and continue as in $G_1$ or $G_2$, respectively. 
Words from (III) are generated by the final two rules. Note that this analysis accounts for \emph{all} words generated by runs beginning with the final four rules from the start symbol $S$.

It remains to verify that the halting protocol cannot be generated by runs beginning with the first rule $S \Rightarrow a YT$.
However, any such run generates only words not beginning with $\#$ and words with a prefix from (I); by definition, no such word could be the halting protocol.    
\end{proof}

We can now define an appropriate \emph{reducing function}:
\begin{equation}
\label{eq:ei-reduction-rnt-to-ch}
    f \colon (\TM,x) \mapsto G_{\TM,x}
\end{equation}
Of course, this function is total and computable.
Moreover it allows us to deduce effective inseparability of $\regtot $ and $\nottot$ from that of $\loops$ and $\halts$:

\begin{proposition}
\label{regtot-nottot-reduces-to-loops-halts}
    The function $f$ in \eqref{eq:ei-reduction-rnt-to-ch} is an e.i.-reduction from $(\regtot, \nottot)$ to $(\loops,\halts)$.
\end{proposition}
\begin{proof}
Clearly $\regtot \cap \nottot = \emptyset$. 
So we need to show the following two properties:
\begin{itemize}
\item if $(\TM,x) \in \mathbf{C}$, then $f((\TM,x)) = G_{\TM,x} \in \mathbf{R}$;
\item if $(\TM,x) \in \mathbf{H}$, then $f((\TM,x)) = G_{\TM,x} \in \nottot$.
\end{itemize}

The second property follows immediately from \cref{grammar-generates-all-but-haltprot}, since the halting protocol cannot be generated by $G_{\TM,x}$. The first property is more involved.

Let $\TM$ loop on $x$ and let $n+1$ be the length of an (incomplete) protocol of $\TM$ on $x$ which reaches a configuration where $\TM$ is in state $q_c$. We shall show that $G_{\TM,x}$ is regularly total, according to the given $n$.
That is, we will show that:
\begin{enumerate}
    \item $S$ generates all nonempty words of length $\leq n+1$.
    \item $Y$ generates all words of length $n$.
\end{enumerate}

For (1), since $\TM$ does not halt on $x$, there is no halting protocol, and $S$ generates all nonempty words, by \cref{grammar-generates-all-but-haltprot}. In particular, $S$ generates all words of length $\leq n + 1$.

For (2), we need to show that any word $\vec a$ of length $n$ belongs to (I)$'$.
For this, even if $\# \vec a$ is a prefix of the (infinite) protocol of $\TM$ on $x$, it includes, by choice of $n$, the letter $q_c$, and so $\# \vec a \notin $(I).
%
%
%
%
\end{proof}

Now we can finally conclude the main result of this section:
\begin{proof}
    [Proof of \cref{regtot-nottot-eff-insep}]
    Follows immediately from \cref{prop-reduction,prop-insepCH,regtot-nottot-reduces-to-loops-halts}.
\end{proof}


\section{Undecidability of (extensions of) $\basicax + \idem$ }
\label{sec:undec-with-idem}

This section is devoted to a proof of our first main result:

\begin{theorem}
\label{thm:undecidability-with-idem}
    Any r.e.\ extension of $\basicax + \idem$ that is sound for $\Lang$ is $\Sigma^0_1$-complete.
\end{theorem}

We shall actually strengthen this result in the next section to avoid the need for idempotency (and even identity), by a sort of reduction, but we present this weaker version to illustrate the main ideas in a simpler setting.

\subsection{Some intermediate lemmata}

To prove \cref{thm:undecidability-with-idem} we first establish a couple intermediate results reasoning about (productive) CFGs in $\basicax (+ \idem)$.
As we shall later consider settings without idempotency and identity, we shall be careful to avoid relying on the multiplicative identity $1$, and shall also be sensitive to the use of idempotency.

\begin{lemma}
\label{lem:membership}
    Fix a CFG $G $ and write $ e_X$ for (any) corresponding solutions in $\basicax$ of each non-terminal $X$ of $G$, cf.~\cref{cfgs-have-solutions-in-basicaxs}. 
    If $X$ generates a word $\vec a$ then $\basicax + \idem \proves \vec a \leq e_X$.
\end{lemma}
\begin{proof}
    By induction on the derivation tree of $\vec a $ from $X$ in $G$.
    Working in $\basicax+\idem$:
    \begin{itemize}
        \item If $X \Rightarrow \vec a$, then we have $\vec a + g = e_X$ for some $g$ by \cref{cfgs-have-solutions-in-basicaxs}, and so $\vec a \leq e_X$ by idempotency of $\vec a$ and \cref{props-of-nat-order}.
        \item If $X \Rightarrow \vec a_0 X_1 \vec a_1 \cdots \vec a_{n-1}  X_n \vec a_n$, s.t.\ $\vec a = \vec a_0 \vec b_1 \vec a_1 \cdots \vec b_n \vec a_n$, then each $\vec b_i $ is generated by $X_i$ by smaller derivation trees.
        Thus by inductive hypothesis we have $\vec b_i \leq e_{X_i}$ for all $i$.
        Finally by monotonicity of products wrt $\leq$, we have indeed $\vec a \leq e_X$.\qedhere
    \end{itemize}
\end{proof}

Let us write $\top \df \mu x \left(\sum\limits_{a \in \Alphabet} a + \sum\limits_{a\in \Alphabet} ax \right)$, which trivially computes the `total' language of all nonempty words.

\begin{lemma}
\label{lem:unfolding-top}
    For $n>0$, we have $\basicax \proves \top = \sum\limits_{|\vec a|=1}^{n}\vec a + \sum\limits_{|\vec a|= n} \vec a \top$.
\end{lemma}

\begin{proof}
    By induction on $n>0$:
    \begin{itemize}
        \item When $n=1$, then the statement follows by just an unfolding of the fixed point $\top$, axiom \eqref{item:fixed-point-axioms}.
        \item For the inductive step we have:
        \[
        \begin{array}{r@{\ = \ }ll}
             \top & \sum\limits_{|\vec a|=1}^{n}\vec a + \sum\limits_{|\vec a|= n} \vec a \top & \text{by inductive hypothesis} \\
                & \sum\limits_{|\vec a|=1}^{n}\vec a + \sum\limits_{|\vec a|= n} \vec a \left(\sum\limits_{a \in \Alphabet} a + \sum\limits_{a\in \Alphabet} a\top \right) & \text{by unfolding $\top$, axiom \eqref{item:fixed-point-axioms}} \\
                & \sum\limits_{|\vec a|=1}^{n}\vec a + \sum\limits_{|\vec a|= n} \vec a \sum\limits_{a \in \Alphabet} a + \sum\limits_{|\vec a|= n} \vec a\sum\limits_{a\in \Alphabet} a\top & \text{by distributivity} \\
                & \sum\limits_{|\vec a|=1}^{n}\vec a + \sum\limits_{|\vec a|= n+1} \vec a + \sum\limits_{|\vec a|= n+1} \vec a\top & \text{by distributivity} \\
                & \sum\limits_{|\vec a|=1}^{n+1}\vec a + \sum\limits_{|\vec a|= n+1} \vec a\top & \text{by rebracketing sums} \qedhere
        \end{array}
        \]
    \end{itemize}
\end{proof}

\subsection{Small proofs of regular totality}

Note that $\top =\mu x \left(\sum\limits_{a \in \Alphabet} a + \sum\limits_{a\in \Alphabet} ax \right)$ is a solution of the non-terminal $T$ in any regularly total grammar.
In what follows, whenever we fix solutions of a regularly total grammar, we assume the solution for $T$ is $\top$.
Again, we shall be sensitive to the use of idempotency, making clear any such instances for our strengthenings in the next section.

\begin{proposition}
\label{reg-tot-provable-with-idem}
    Fix a productive CFG $G \in \mathbf R$ and let $e_X$ be solutions of each if its non-terminals $X$ in $\basicax$, cf.~\cref{cfgs-have-solutions-in-basicaxs}, with $e_T = \top$. 
    Then $\basicax + \idem \proves \top \leq e_S$.
\end{proposition}
\begin{proof}
Since $G\in \mathbf R$, let $n>0$ be such that:
\begin{enumerate}
    \item\label{item:S-gens-all-short-words} $S$ generates all nonempty words of length $\leq  n+1$.
    \item\label{item:Y-gens-bar} $Y$ generates all words of length $n$.
    \item\label{item:G-reaches-loop} $G$ has a production $S \to aYT$ for each $a \in \Alphabet$.
\end{enumerate}
    Let us work in $\basicax + \idem$. 
    By \eqref{item:Y-gens-bar} above and \cref{lem:membership} we have
    $ \vec a \leq e_Y$ whenever $|\vec a| = n$.
    By \purple{idempotency of $e_Y$} and \cref{intro-elim-props-for-+} we have $\sum\limits_{|\vec a| = n}\vec a \leq e_Y $.
    Thus by monotonicity of $\cdot$ and distributivity:
    $$\sum\limits_{|\vec a| = n}\vec a \top \leq e_Y\top$$
    Now by \eqref{item:G-reaches-loop} and \cref{cfgs-have-solutions-in-basicaxs} we have $g + \sum\limits_{a \in \Alphabet}ae_Y\top = e_S$ for some $g$ (and since $e_T$ is $\top$).
    So by \purple{idempotency of $\sum\limits_{a \in \Alphabet}ae_Y\top$} we have:
    $$\sum\limits_{a \in \Alphabet} ae_Y\top \leq e_S$$
    Putting the two displays above together we have $\sum\limits_{a \in \Alphabet} a\sum\limits_{|\vec a| = n}\vec a \top \leq e_S$ by monotonicity, and so by distributivity:
    \begin{equation}
    \label{eq:long-words-reach-S}
        \sum\limits_{|\vec a| = n+1}\vec a \top \leq e_S
    \end{equation}
    
    Now by \eqref{item:S-gens-all-short-words} and again by \cref{lem:membership} we have $\vec a \leq e_S$ whenever $1 \leq |\vec a|\leq n+1$. 
    Again by \purple{idempotency of $e_S$} and monotonicity of $+$ we have 
    $\sum\limits_{|\vec a|=1}^{n+1} \vec a \leq e_S$.
    Putting this together with \eqref{eq:long-words-reach-S} above, we have \purple{by idempotency of $e_S$} and \cref{intro-elim-props-for-+}:
    $$\sum\limits_{|\vec a|=1}^{n+1} \vec a  + \sum\limits_{|\vec a| = n+1}\vec a \top \leq e_S$$
    From here we conclude by \cref{lem:unfolding-top} and transitivity of $\leq$, \cref{props-of-nat-order}.
\end{proof}

\subsection{Putting it together}
We are now ready to prove the main result of this section:

\begin{proof}
    [Proof of \cref{thm:undecidability-with-idem}]
    Let us assume that all CFGs have start symbol $S$ and write $e^G_X$ for the solution of the nonterminal $X$ in $G$ obtained by \cref{cfgs-have-solutions-in-basicaxs}.
    Let $\mathsf U \supseteq \basicax+\idem$ be r.e.\ with $\Lang \models \mathsf U$.  
    Writing $\total {\mathsf U} \df \{G \text{ productive} : \mathsf U \proves \top \leq e^G_S \}$,
    we have:
    \begin{itemize}
        \item $\total {\mathsf U}  \cap \nottot = \emptyset$, by assumption that $\Lang \models \mathsf U$.
        \item $\mathbf R \subseteq \total {\mathsf U}  $, since $\basicax+\idem \subseteq \mathsf U$ and by \cref{reg-tot-provable-with-idem}.
    \end{itemize}
    Thus by \cref{separator-of-eff-insep-sets-is-S01-complete,regtot-nottot-eff-insep} we have that $\total {\mathsf U}$ is $\Sigma^0_1$-complete.
    Since there is an obvious computable reduction from $\total {\mathsf U}$ to $\mathsf U$, we have that $\mathsf U$ is $\Sigma^0_1$-complete too, as required.
\end{proof}

Since all theories we have introduced are r.e.\ and modelled by $\Lang$, we thus immediately have:
\begin{corollary}
    The equational theories of idempotent Conway $\mu$-semirings, idempotent Park $\mu$-semirings and Chomsky algebras are undecidable, and moreover $\Sigma^0_1$-complete.
\end{corollary}

\section{The case without idempotency}
\label{sec:undec-without-idem}

In this section we shall strengthen the main result of the previous section to cover theories of semirings that are not necessarily idempotent.



    



\subsection{Identity-free $\mu$-semirings and productive expressions}

Write $\basicaxidfree$ for the axiomatisation defined just like $\basicax$ but without any of the axioms mentioning $1$ (and without $1$ in the language).

The \defname{productive} expressions (in the language of $\basicaxidfree$), written $p,q$ etc., are generated by,
\begin{equation}
    \label{eq:productive-expressions}
    p,q,\dots \quad ::= \quad 0 \quad \mid \quad a 
\quad \mid \quad p+q 
\quad \mid \quad pe \quad \mid \quad ep
\quad \mid \quad \mu X p
\end{equation}
where $e$ ranges over (possibly open, not necessarily productive) $1$-free expressions.

By inspection of the proof of \cref{eq-systems-have-solns-in-basicax}, we can give a more refined result than \cref{cfgs-have-solutions-in-basicaxs}:

\begin{proposition}
\label{prod-cfg-has-prod-solns}
    A productive CFG has productive solutions in $\basicaxidfree$.
    I.e., given a productive CFG $\{X_i \Rightarrow f_{ij}(\vec X)\}_{i<n,j}$, there are productive expressions $\vec e = e_0, \dots, e_{n-1}$ s.t.\ $\basicaxidfree \proves e_i = \sum_j f_{ij}(\vec e)$ for all $i<n$.
\end{proposition}
In particular note that the RHS of a production in a productive CFG is a productive expression, whence the same proof of \cref{eq-systems-have-solns-in-basicax} applies for the statement above.
Note also that that proof does not rely on any properties of the multiplicative identity $1$ (which in any case plays no role in productive CFGs).

\subsection{An idempotent interpretation of productive expressions}

Temporarily write $e':= \mu x (e + \mu y (x + y))$. 
It is not hard to see that, over $\basicaxidfree$, $e'$ is an idempotent:
\[
\begin{array}{r@{\ = \ }ll}
      e'  & e + \mu y (e' + y) & \text{by unfolding $e'$, by axiom \eqref{item:fixed-point-axioms}} \\
        & e + e' + \mu y (e' + y) & \text{by unfolding $\mu y (e' + y)$, by axiom \eqref{item:fixed-point-axioms}} \\
        & (e + \mu y (e' + y) ) + e' & \text{by associativity and commutativity of $+$}\\
        & e' + e' & \text{by the first line}
\end{array}
\]

In the presence of an order $e'$ is also an upper bound of $e$, and in the presence of induction $e'$ is in fact the \emph{least} idempotent upper bound of $e$.
In particular in $\Lang $ (and any other idempotent model) we have simply $e'=e$:
\begin{itemize}
    \item On one hand, we have $\lang {e'} = \lang{e + \mu y (e'+y)}$, and so $\lang {e'} \supseteq \lang e$.
    \item On the other hand, we have that $e$ solves the matrix of $e'$:
    \begin{itemize}
        \item Clearly $\lang {\mu y (e+y)} = e$: as above we have $\lang {\mu y (e+y)} = \lang{e + \mu y (e+y)} \supseteq \lang e$; we also have that $e$ solves the matrix of $\mu y (e+y)$ in $\Lang$, i.e., $\lang e = \lang {e+e}$, thanks to idempotency.
        \item Thus we have $e = e + e = e + \mu y (e+y)$ in $\Lang$, again by idempotency.
    \end{itemize}
    Thus indeed $\lang {e'}\subseteq \lang e$, as $e'$ must be a least solution.
\end{itemize}

\begin{definition}
    [Idempotent interpretation]
    For $1$-free expressions $e$, define $\id e $ as follows:
\begin{itemize}
    \item $\id a \df a' = \mu x (a + \mu y (x+y))$
    \item $\id \cdot $ commutes with everything else. I.e.,
    \[
    \begin{array}{r@{\ \df\ }l}
         \id 0 & 0 \\
         \id x & x \\
         \id {(e+f)} & \id e + \id f \\
         \id {(ef)} & \id e \id f \\
         \id{(\mu x e)} & \mu x \id e
    \end{array}
    \]
\end{itemize}
\end{definition}

As previously argued it is not hard to see that:
\begin{proposition}
\label{idem-trans-equiv-in-lang}
    $\lang {\id e} = \lang {e}$.
\end{proposition}
\begin{proof}
[Proof sketch]
    By induction on the structure of $e$. In fact, we must strengthen the statement to $\lang {\id e(\vec A)} = \lang {e(\vec A)}$ for all languages $\vec A$.
    The case when $e$ is some $a \in \Alphabet$ follows from the argument above.
    Every other case is immediate, with the $\mu$-case following by the strong $\mu$-congruence principle in $\Lang$: $\forall x\,    e(x) =f(x) \limp \mu x e(x) = \mu x f(x)$.
\end{proof}

Moreover, $\id \cdot$ induces a bona fide interpretation of identity-free $\mu$-semirings:

\begin{proposition}
\label{idem-trans-is-an-interp}
    $\basicaxidfree\proves e=f \implies \basicaxidfree \proves \id e = \id f$.
\end{proposition}
\noindent
To see this, simply replace every expression $e $ in a $\basicaxidfree$ proof by $\id e$. It remains to verify that the axioms and rules remain correct, which is routine.

Finally, thanks to the fact that idempotent elements of a semiring form an ideal, cf.~\cref{rem:idempotents-form-ideal}, we have idempotency of the image of $\id \cdot $ on productive expressions provably in $\basicaxidfree$:

\begin{lemma}
\label{idem-trans-is-idem-on-prod-exprs}
    If $e(\vec x)$ is productive then $\basicaxidfree \proves \id e(\vec f) + \id e(\vec f) = \id e(\vec f)$ for all $\vec f$.
\end{lemma}
\begin{proof}
    By induction on the structure of $e(\vec x)$, according to the grammar of productive expressions in \eqref{eq:productive-expressions}.
    Working in $\basicaxidfree$ we have:
    \begin{itemize}
        \item If $e(\vec x)$ is $0$ then: 
        \[
        \begin{array}{r@{\ = \ }ll}
             \id 0 + \id 0 & 0 + \id 0 & \text{by definition of $\id 0$} \\
                & \id 0 & \text{since $0$ is an additive unit}
        \end{array}
        \]
        \item If $e(\vec x)$ is some $a\in \Alphabet$ then $\id a + \id a = \id a$ by the argument at the beginning of this subsection.
        \item If $e(\vec x)$ is $p(\vec x) +  q(\vec x) $ then:
        \[
        \begin{array}{rll}
        & (\id p(\vec f) + \id q(\vec f)) + (\id p(\vec f) + \id q(\vec f)) & 
        \\
          =& (\id p(\vec f) + \id p(\vec f)) + (\id q(\vec f) + \id q(\vec f)) & \text{by associativity and commutativity of $+$} \\
            =& \id p(\vec f) + \id q(\vec f) & \text{by inductive hypothesis for $p(\vec x)$ and $q(\vec x)$} 
        \end{array}
        \]
        \item If $e(\vec x)$ is $p(\vec x)g(\vec x)$ then:
        \[
        \begin{array}{rll}
        &  \id p(\vec f)\id g(\vec f) + \id p(\vec f) \id g(\vec f) & 
        \\
          =  & (\id p(\vec f) + \id p(\vec f))\id g(\vec f) & \text{by distributivity} \\
           = & \id p(\vec f) \id g(\vec f) & \text{by inductive hypothesis for $p(\vec x)$} 
        \end{array}
        \]
        \item If $e(\vec x)$ is $g(\vec x)p(\vec x)$ then the argument is symmetric to above.
        \item If $e(\vec x)$ is $\mu x p(x,\vec x)$ then:
        \[
        \begin{array}{rll}
             & \mu x \id p(x,\vec f) + \mu x \id p(x,\vec f) 
             \\
              =  & \id p (\mu x \id p(x,\vec f),\vec f) + \id p (\mu x \id p(x,\vec f), \vec f) & \text{by $\mu$-unfolding, axiom \eqref{item:fixed-point-axioms}} \\
              = & \id p (\mu x \id p(x,\vec f),\vec f) & \text{by inductive hypothesis for $p(x,\vec x)$} \qedhere
        \end{array}
        \]
    \end{itemize}
\end{proof}

Note that, the above result notwithstanding, $\id e$ is typically not productive, even if $e$ is. In particular $\id a$ is not productive, even though $a$ is.

    Notice that, \cref{idem-trans-is-an-interp} notwithstanding, $\id \cdot$ does not necessarily interpret $\basicaxidfree + \idem$ into $\basicaxidfree$. This is because, a priori, a proof in $\basicaxidfree+\idem$ may exploit idempotency of expressions not in the image of $\id \cdot $ on productive expressions.
    For this reason, in what follows,
    we shall exploit the fact that we carefully kept track of the use of idempotency in our arguments in the previous section.

\subsection{Undecidability of $\basicaxidfree$ (and extensions)}
Recall $\top \df \mu x\left(\sum \Alphabet + \sum\limits_{a\in \Alphabet} ax\right)$.
By a similar argument to \cref{reg-tot-provable-with-idem} in the previous section we can obtain:

\begin{proposition}
\label{reg-tot-prov-under-idem-trans}
    Let $G \in \regtot$ and let $e_X$ be productive solutions of each of its non-terminals $X$ in $\basicaxidfree$, cf.~\cref{prod-cfg-has-prod-solns}, with $e_T = \top$. 
    Then $\basicaxidfree \proves \id\top \leq \id{e_S}$. 
\end{proposition}

For this we need appropriate versions of the two intermediate results we employed, \cref{lem:membership,lem:unfolding-top}, in the previous section.
Their proofs are morally similar to the previous ones, but we must be careful to state them appropriately in the absence of idempotency, exploiting the refined results of this section on productive CFGs and $\basicaxidfree$
\begin{lemma}
\label{lem:membership-productive}
    Fix a productive CFG $G$ and write $e_X$ for (any) corresponding productive solutions in $\basicaxidfree$ of each non-terminal $X$ of $G$, cf.~\cref{prod-cfg-has-prod-solns}.
    If $\vec a \in \lang X$ in $G$ then $\basicaxidfree \proves \id{\vec a} \leq \id e_X$.
\end{lemma}
\begin{proof}
    By induction on the derivation tree of $\vec a $ from $X$ in $G$. Working in $\basicaxidfree$:
    \begin{itemize}
        \item If $X\to \vec a $ then we have $\vec a + g = e_X$ for some $g$ by \cref{prod-cfg-has-prod-solns}.
        Thus $\id{\vec a} + \id g = \id e_X$ by \cref{idem-trans-is-an-interp}, and so $\id{\vec a} \leq \id{e_X}$ by idempotency of $\id{\vec a}$, \cref{idem-trans-is-idem-on-prod-exprs}.
        \item The inductive step is the same as in the proof of \cref{lem:membership}. \qedhere
    \end{itemize}
\end{proof}

\begin{lemma}
\label{lem:unfolding-top-no-id}
    For $n>0$, we have $\basicaxidfree \proves \id\top = \sum\limits_{|\vec a|=1}^{n}\id{\vec a} + \sum\limits_{|\vec a|= n} \id{\vec a} \id\top$.
\end{lemma}
\begin{proof}
    The proof of \cref{lem:unfolding-top} does not use any identity axioms, and so goes through already in $\basicaxidfree$.
    The result now follows by \cref{idem-trans-is-an-interp}.
\end{proof}

We are now ready to prove the result we stated at the beginning of this section:

\begin{proof}
    [Proof of \cref{reg-tot-prov-under-idem-trans}]
    The proof goes through just like the proof of \cref{reg-tot-provable-with-idem}, only replacing all (productive) expressions by their image under $\id \cdot$.
    All appeals to \cref{lem:membership,lem:unfolding-top} are replaced by \cref{lem:membership-productive,lem:unfolding-top-no-id} respectively.
    Note that, by inspection of the proof, the only uses of idempotency (marked in \purple{purple}) are on productive expressions, whose image under $\id \cdot$ is indeed provably idempotent in $\basicaxidfree$ by \cref{idem-trans-is-idem-on-prod-exprs}.
    In particular, the expressions $e_Y$ and $e_S$ are productive by assumption.
\end{proof}

\subsection{Putting it together}
Finally we arrive at our promised strengthening of \cref{thm:undecidability-with-idem}:
\begin{theorem}
\label{thm:undecidability-without-idem}
    Any r.e.\ extension of $\basicaxidfree$ sound for $\Lang$ is $\Sigma^0_1$-complete.
\end{theorem}
The proof is similar to that of \cref{thm:undecidability-with-idem}, only being more careful about handling the $\id\cdot$ translation.
\begin{proof}
    Like in the proof of \cref{thm:undecidability-with-idem}, let us assume that all CFGs have start symbol $S$ and write $e^G_X$ for the productive solutions of nonterminals $X$ in a productive grammar $G$ obtained by \cref{prod-cfg-has-prod-solns}.
    Let $\mathsf U \supseteq \basicaxidfree$ be r.e.\ with $\Lang \models \mathsf U$.
    Now we write $\total{\mathsf U} \df \{G \text{ productive}: \mathsf U \proves \id \top \leq \id{(e^G_S)}\}$. We have:
    \begin{itemize}
        \item $\total {\mathsf U} \cap \nottot = \emptyset$, by assumption that $\Lang \models \mathsf U $ and by \cref{idem-trans-equiv-in-lang}.
        \item $\regtot \subseteq \total{\mathsf U}$, since $\basicaxidfree\subseteq \mathsf U$ and by \cref{reg-tot-prov-under-idem-trans}.
    \end{itemize}
    Thus by \cref{separator-of-eff-insep-sets-is-S01-complete,regtot-nottot-eff-insep} we have that $\total{\mathsf U}$ is $\Sigma^0_1$-complete.
    Again, by the obvious computable reduction from $\total{\mathsf U}$ to $\mathsf U$, we conclude as required. 
\end{proof}

\begin{corollary}
    The equational theories of Conway $\mu$-semirings and Park $\mu$-semirings are undecidable, and moreover $\Sigma^0_1$-complete.
    The same holds even for identity-free versions of these theories.
\end{corollary}

\section{Conclusions}
\label{sec:concs}
In this work we addressed the question of (un)decidability of the equational theories of various classes of semirings with fixed points. 
Using recursion theoretic methods we showed that a large class of such theories are undecidable.
These results account for many of the $\mu$-semiring theories proposed in previous works, cf.~\cite{Leiss92:ka-with-recursion,EsikLeiss02,EsikLeiss05,GHK13}.

Conway semirings were proposed in \cite{EsikLeiss02,EsikLeiss05} as a basic finite equational theory that is nonetheless powerful enough to carry out reduction to Greibach normal form.
Our basic system $\basicax$ is weaker/more general than Conway semirings, and is also finitely equationally axiomatised. 
It would be interesting to examine if it may have any further utility, other than as a lower bound for the range of our undecidability results herein.
For instance, are there results of formal language theory that are validated already in $\basicax$?

Finally our results do not cover pure $\mu$-semirings, where $\mu x e(x)$ is not necessarily a fixed point. The (un)decidability of its equational theory is, as far as we know, open.

\section*{Acknowledgements}
The (alphabetically) first two authors have been supported by a UKRI Future Leaders Fellowship, Structure vs
Invariants in Proofs, project number MR/S035540/1. The work of the third author was performed at the Steklov International Mathematical Center and supported by the Ministry of Science and Higher Education of the Russian Federation (agreement no. 075-15-2025-303).

\bibliographystyle{alpha}
\bibliography{kleene}

\end{document}